\documentclass[12pt]{article}
\usepackage{amsmath,latexsym,amssymb}
\usepackage{amsthm}
\usepackage{enumerate}
\usepackage{epsfig}

\newcommand\NN{\mathbb{N}}

\newcommand\ZZ{\mathbb{Z}}
\newcommand\eps{{\varepsilon}}
\newcommand\dist{\operatorname{dist}}

\newcommand{\address}{Address: Department of Mathematics, University of North Texas, 1155 Union Circle \#311430, Denton, TX 76203-5017, USA; E-mail: allaart@unt.edu}

\newtheorem{theorem}{Theorem}

\newtheorem{lemma}[theorem]{Lemma}
\newtheorem{remark}[theorem]{Remark}

\title{Digital sum inequalities and approximate convexity of Takagi-type functions} 

\author{Pieter C. Allaart \\
University of North Texas \footnote{\address}}
\date{\today}

\begin{document}

\maketitle

\begin{abstract}
For an integer $b\geq 2$, let $s_b(n)$ be the sum of the digits of the integer $n$ when written in base $b$, and let $S_b(N)=\sum_{n=0}^{N-1}s_b(n)$. Several inequalities are derived for $S_b(N)$. Some of the inequalities can be interpreted as comparing the average value of $s_b(n)$ over integer intervals of certain lengths to the average value of a beginning subinterval. Two of the main results are applied to derive a pair of ``approximate convexity" inequalities for a sequence of Takagi-like functions. One of these inequalities was discovered recently via a different method by V. Lev; the other is new.



\end{abstract}

\section{Introduction}

Fix an integer $b\geq 2$, and for $n\in\NN$, write the $b$-ary representation of $n$ as $n=\sum_{j=0}^\infty \alpha_j(n)b^j$, where $\alpha_j(n):=\alpha_j(n;b)\in\{0,1,\dots,b-1\}$ for each $j$. Define the $b$-ary digital sum and cumulative $b$-ary digital sum respectively by
\begin{equation*}
s_b(n)=\sum_{j=0}^\infty \alpha_j(n), \qquad n\in\ZZ_+,
\end{equation*}
and
\begin{equation*}
S_b(N)=\sum_{n=0}^{N-1}s_b(n), \qquad N\in\ZZ_+,
\end{equation*}
where $\ZZ_+$ denotes the set of nonnegative integers, and we make the usual convention that the empty sum is equal to zero. These digital sums have been well investigated in the literature, especially for the case $b=2$. The investigations have mainly focused in two directions: finding exact or asymptotic formulas for $S_b(N)$ (e.g. Trollope \cite{Trollope}, Delange \cite{Delange}) or determining the probability distribution of $s_b(n)$ as $n$ ranges over certain subsets of the positive integers (e.g. Mauduit and S\'ark\"ozy \cite{Mauduit}, Rivat \cite{Rivat} or Drmota, Mauduit and Rivat \cite{Drmota}, among many others). Stolarsky \cite{Stolarsky} discusses a wide range of applications of digital sums. The aim of the present article is to prove a number of inequalities for $S_b(N)$. One of these inequalities, for the ternary case, came about naturally in the author's quest to find a simpler proof of a recent result of Lev \cite{Lev} concerning the ``approximate convexity" of a particular continuous but nowhere differentiable function akin to the Takagi function. The other results all concern general $b$. They are either needed in the proof of the above-mentioned inequality, or are further developments of special cases of it.
Some of the inequalities are most elegantly stated in terms of the average values
\begin{equation}
\bar{s}_b(s,t):=\frac{S_b(t)-S_b(s)}{t-s}=\frac{1}{t-s}\sum_{n=s}^{t-1}s_b(n), \qquad 0\leq s<t.
\label{eq:averages}
\end{equation}
The inequalities of Theorems \ref{thm:general-bound} and \ref{thm:b-times-longer} below compare the average value of $s_b(n)$ over certain intervals of integers to the average value over a beginning subinterval.

The first inequality is in effect a strong form of superadditivity. It is known for the case $b=2$; see, for instance, section 4 of McIlroy \cite{McIlroy}, where the inequality is used to determine the extremal cost in a merging process.

\begin{theorem} \label{thm:-strong-super-additivity}
For any nonnegative integers $n$ and $m$, we have
\begin{equation}
S_b(m+n)\geq S_b(m)+S_b(n)+\min\{m,n\}.
\label{eq:super-addivity}
\end{equation}
\end{theorem}

Theorem \ref{thm:-strong-super-additivity} is used to prove the following result, which specializes to the case $b=3$ and is a number-theoretic version of Theorem 3 of Lev \cite{Lev}. See Section \ref{sec:application}, where this connection is outlined in detail.

\begin{theorem} \label{thm:ternary-inequality}
For any integers $k$, $l$ and $m$ with $0\leq l\leq k\leq m$, we have
\begin{equation}
S_3(m+k+l)+S_3(m-k)+S_3(m-l)-3S_3(m)\leq 2k+l.
\label{eq:ternary-inequality}
\end{equation}
\end{theorem}

Two special cases of the above inequality are particularly interesting: the case $l=0$ and the case $l=k$. For $l=0$, \eqref{eq:ternary-inequality} reduces to
\begin{equation*}
S_3(m+k)+S_3(m-k)-2S_3(m)\leq 2k.
\end{equation*}
This inequality holds in fact with strict inequality, and the factor 2 on the right can not be replaced by any smaller number. These observations follow from the following, more general result.

\begin{theorem} \label{thm:general-bound}
Let $b\geq 2$ be arbitrary. 
\begin{enumerate}[(i)]
\item For any nonnegative integers $k$ and $m$ with $k\leq m$, we have
\begin{equation}
S_b(m+k)+S_b(m-k)-2S_b(m)\leq \left[\frac{b+1}{2}\right]k,
\label{eq:general-bound}
\end{equation}
where $[x]$ denotes the greatest integer less than or equal to $x$. 
The constant $[(b+1)/2]$ can not be replaced by a smaller constant. However, strict inequality holds in \eqref{eq:general-bound} when $b$ is odd.
\item For any nonnegative integers $n$ and $k$, we have
\begin{equation}
\bar{s}_b(n,n+2k)\leq \bar{s}_b(n,n+k)+\frac12\left[\frac{b+1}{2}\right].
\label{eq:general-bound-av}
\end{equation}
\end{enumerate}
\end{theorem}

(For the case $b=2$, this result was proved previously by the present author; see \cite{Allaart}.)

\bigskip
On the other extreme, the case $l=k$ of \eqref{eq:ternary-inequality} simplifies to
\begin{equation*}
S_3(m+2k)+2S_3(m-k)-3S_3(m)\leq 3k.
\end{equation*}
Equality obtains when $k=m$ (see Lemma \ref{lem:base-multiply} below).
Setting $n=m-k$ and dividing by $3k$, the last inequality can be written as
\begin{equation*}
\bar{s}_3(n,n+3k)\leq \bar{s}_3(n,n+k)+1.
\end{equation*}
This extends to arbitrary $b\geq 2$ as in the following theorem, which states that the average value of $s_b(n)$ over any integer interval of length $bk$ is at most $(b-1)/2$ greater than the average over the first $k$ integers in the interval.

\begin{theorem} \label{thm:b-times-longer}
For each $b\geq 2$ and for all $n,k\geq 0$, we have
\begin{equation}
\bar{s}_b(n,n+bk)\leq \bar{s}_b(n,n+k)+\frac{b-1}{2}.
\label{eq:times-b-average}
\end{equation}
Moreover, equality obtains for each $k$ when $n=0$.
\end{theorem}

Note that for $b=2$, \eqref{eq:general-bound-av} and \eqref{eq:times-b-average} give the same result.

The proofs of Theorems \ref{thm:-strong-super-additivity}-\ref{thm:b-times-longer} are given in the next section. Two of the theorems are then used in Section \ref{sec:application} to derive a pair of inequalities for a sequence of Takagi-like functions.

\section{Proofs of the main results}

Throughout this section, let $b\geq 2$ be fixed. It is convenient to introduce the notation
\begin{equation}
\Sigma_b(s,t):=\sum_{r=s}^{t-1} s_b(r)=S_b(t)-S_t(s), \qquad s<t.
\label{eq:Sigma-notation}
\end{equation}
Thus, $\Sigma_b(s,t)$ is the sum of all the $b$-ary digits needed to write the block of consecutive integers $s,s+1,\dots,t-1$. When there is no confusion possible about the base $b$, the subscript $b$ will be frequently dropped throughout this paper.

We first state a useful lemma.

\begin{lemma} \label{lem:complementation}
For any nonnegative integers $p, j, k$ and $n$ with $0\leq k\leq n\leq jb^p$ and $j\leq b$,
\begin{equation*}
\Sigma(jb^p-k,jb^p)-\Sigma(n-k,n)=\Sigma(jb^p-n,jb^p-n+k)-\Sigma(0,k).
\end{equation*}
\end{lemma}

\begin{proof}
This follows at once since $s_b(jb^p-r-1)+s_b(r)=(b-1)p+j-1$, independent of $r$, for $0\leq r<jb^p$.
\end{proof}

We will also use the following, easily verified fact: for any nonnegative integers $n$ and $k$, $s_b(n+b^k)\leq s_b(n)+1$. Applying this repeatedly, we obtain the useful estimate
\begin{equation}
s_b\left(n+\sum_{i=1}^k b^{p_i}\right)\leq s_b(n)+k, \qquad n\in\ZZ_+, \quad p_1,\dots,p_k\in\ZZ_+.
\label{eq:adding-r-powers}
\end{equation}

\begin{proof}[Proof of Theorem \ref{thm:-strong-super-additivity}]
The statement is obvious for the case $m=n=0$. We proceed by induction on $m+n$. Let $N\in\NN$, and assume 
\eqref{eq:super-addivity} holds for all pairs $(m,n)$ with $m+n<N$. Suppose $m$ and $n$ are such that $m+n=N$. By symmetry we may assume that $m\geq n$. 
In terms of the notation \eqref{eq:Sigma-notation}, we must show that
\begin{equation*}
\Sigma(m,m+n)\geq \Sigma(0,n)+n.
\end{equation*}
This is trivial when $n=0$, so assume $n\geq 1$.
We consider two cases:

\bigskip
{\em Case 1.} The range $\{m+1,\dots,m+n-1\}$ does not contain a power of $b$. In this case, there is $p\in\ZZ_+$ such that $b^p\leq m\leq m+n-1<b^{p+1}$. So we can subtract 1 from the first digit of each number $m,\dots,m+n-1$ and obtain
\begin{align*}
\Sigma(m,m+n)&=n+\Sigma(m-b^p,m+n-b^p)\\
&=n+S(m+n-b^p)-S(m-b^p)\\
&\geq n+S(n)+\min\{m-b^p,n\}\\
&\geq n+\Sigma(0,n).
\end{align*}

{\em Case 2.} The range $\{m+1,\dots,m+n-1\}$ contains a power of $b$; say $m+j=b^p$, where $1\leq j<n$. Since subtracting $b^p$ maps $\{m+j,\dots,m+n-1\}$ onto $\{0,\dots,n-j-1\}$, we see that
\begin{equation}
\Sigma(m+j,m+n)=\Sigma(0,n-j)+n-j.
\label{eq:first-part}
\end{equation}
On the other hand, by Lemma \ref{lem:complementation},
\begin{equation*}
\Sigma(m,m+j)-\Sigma(n-j,n)=\Sigma(m+j-n,m+2j-n)-\Sigma(0,j).
\end{equation*}
Since $j<n$, the induction hypothesis implies
\begin{align*}
\Sigma(m+j-n,m+2j-n)&=S((m+j-n)+j)-S(m+j-n)\\
&\geq S(j)+j=\Sigma(0,j)+j.
\end{align*}
Hence,
\begin{equation}
\Sigma(m,m+j)-\Sigma(n-j,n)\geq j.
\label{eq:second-part}
\end{equation}
Combining \eqref{eq:first-part} and \eqref{eq:second-part} yields
\begin{align*}
\Sigma(m,m+n)-\Sigma(0,n)&=\Sigma(m+j,m+n)+\Sigma(m,m+j)-\Sigma(0,n)\\
&=\Sigma(m,m+j)-\Sigma(n-j,n)+n-j\\
&\geq j+(n-j)=n,
\end{align*}
as required.
\end{proof}

\begin{remark}
{\rm
The inequality \eqref{eq:super-addivity} is sharp in the sense that equality holds whenever $n$ is a power of $b$ and $m<n$.
}
\end{remark}

The following identity is well known for the case $b=2$; see McIlroy \cite[eq. (4a)]{McIlroy}.

\begin{lemma} \label{lem:base-multiply}
For each $m\in\NN$,
\begin{equation*}
S_b(bm)=bS(m)+\frac{b(b-1)m}{2}.
\end{equation*}
\end{lemma}

\begin{proof}
For each number $j\in\{0,\dots,m-1\}$ and $r\in\{0,\dots,b-1\}$, $s_b(bj+r)=s_b(j)+r$. Summing over $r$ and then over $j$ gives the lemma.
\end{proof}

\begin{proof}[Proof of Theorem \ref{thm:ternary-inequality}]
Note first that \eqref{eq:ternary-inequality} can be stated equivalently as
\begin{equation}
\Sigma(m,m+k+l)-\Sigma(m-k,m)-\Sigma(m-l,m)\leq 2k+l,
\label{eq:ternary-inequality-sigma-form}
\end{equation}
where the omitted subscript is understood to be $b=3$.
We use induction on the sum $m+k+l$. The statement is trivial for all $m$ when $k=l=0$. Let $N\in\NN$, and assume \eqref{eq:ternary-inequality-sigma-form} holds whenever $m+k+l<N$. Suppose $(k,l,m)$ is a triple with $0\leq l\leq k\leq m$ and $m+k+l=N$.
If $m\leq 2(k+l)$, then $2m-k-l\leq m+k+l$ and so a double application of Theorem \ref{thm:-strong-super-additivity} gives
\begin{align*}
S(3m)&\geq S(m+k+l)+S(2m-k-l)+(2m-k-l)\\
&\geq S(m+k+l)+S(m-k)+S(m-l)+(m-k)+(2m-k-l)\\
&=S(m+k+l)+S(m-k)+S(m-l)+3m-(2k+l).
\end{align*}
On the other hand, $S(3m)=3S(m)+3m$ by Lemma \ref{lem:base-multiply}, and combining these results gives \eqref{eq:ternary-inequality}. In the remainder of the proof we may therefore assume that $m>2(k+l)$. Since $l\leq k$, this implies that
\begin{equation}
m+k+l<2(m-l),
\label{eq:at-most-double}
\end{equation}
and
\begin{equation}
m+k+l<3(m-k).
\label{eq:at-most-triple}
\end{equation}
Hence, the range $\{m-k,\dots,m+k+l-1\}$ contains at most one power of $3$.

\bigskip
\noindent {\em Case 1.} The range $\{m-k+1,\dots,m-1\}$ does not contain a power of $3$. Then there is $i\in\{1,2\}$ and $p\in\ZZ_+$ such that $3^p i\leq m-k\leq m-1<3^p(i+1)$, so the numbers $m-k,\dots,m-1$ all have leading ternary digit $i$. Hence,
\begin{equation*}
\Sigma(m-k,m)=\Sigma(m-k-3^p i,m-3^p i)+ki,
\end{equation*}
and likewise,
\begin{equation*}
\Sigma(m-l,m)=\Sigma(m-l-3^p i,m-3^p i)+li.
\end{equation*}
On the other hand, for each $n\in\NN$ we have $s(n+3^p i)\leq s(n)+i$ in view of \eqref{eq:adding-r-powers}, and therefore
\begin{equation*}
\Sigma(m,m+k+l)\leq\Sigma(m-3^p i,m+k+l-3^p i)+(k+l)i.
\end{equation*}
Hence, setting $m'=m-3^p i$, we have
\begin{align*}
\Sigma(m,m+k+l)-&\Sigma(m-k,m)-\Sigma(m-l,m)\\
&\leq \Sigma(m',m'+k+l)-\Sigma(m'-k,m')-\Sigma(m'-l,m')\\
&\leq 2k+l,
\end{align*}
where the last inequality uses the induction hypothesis.

\bigskip
\noindent {\em Case 2.} The range $\{m-k+1,\dots,m-1\}$ contains a power of $3$. Say $m-j=3^p$, where $0<j<k$. We consider two subcases:

\bigskip
{\em Case 2a.} The power of $3$ is among $m-k+1,\dots,m-l$, so $l\leq j<k$. By the induction hypothesis (with $j$ in place of $k$),
\begin{equation}
\Sigma(m,m+j+l)-\Sigma(m-j,m)-\Sigma(m-l,m)\leq 2j+l.
\label{eq:first-half}
\end{equation}
Next, placing a digit ``2" in front of the numbers $m-k,\dots,m-j-1$ increases their digital sums by exactly 2, so that
\begin{equation*}
\Sigma(m-k,m-j)=\Sigma(m-k+2\cdot 3^p,m-j+2\cdot 3^p)-2(k-j).
\end{equation*}
By \eqref{eq:at-most-triple}, the numbers $m+j+l,\dots,m+k+l-1$ are strictly smaller than $3^{p+1}$. Hence, by Theorem \ref{thm:-strong-super-additivity} and Lemma \ref{lem:complementation},
\begin{equation*}
\Sigma(m-k+2\cdot 3^p,m-j+2\cdot 3^p)\geq \Sigma(m+j+l,m+k+l),
\end{equation*}
and so
\begin{equation}
\Sigma(m-k,m-j)\geq \Sigma(m+j+l,m+k+l)-2(k-j).
\label{eq:second-half}
\end{equation}
Combining \eqref{eq:first-half} and \eqref{eq:second-half} gives \eqref{eq:ternary-inequality-sigma-form}.

\bigskip
{\em Case 2b.} The power of $3$ is among $m-l+1,\dots,m-1$, so $0<j<l$. By the induction hypothesis (with $j$ in place of $l$),
\begin{equation}
\Sigma(m,m+k+j)-\Sigma(m-k,m)-\Sigma(m-j,m)\leq 2k+j.
\label{eq:2b-first-half}
\end{equation}
Now by \eqref{eq:at-most-double}, the first digit of $m+k+l-1$ must be a ``1". We can now place a ``1" in front of each number $m-l,\dots,m-j-1$ and use Theorem \ref{thm:-strong-super-additivity} and Lemma \ref{lem:complementation} to obtain
\begin{align*}
\Sigma(m-l,m-j)&=\Sigma(m-l+3^p,m-j+3^p)-(l-j)\\
&\geq \Sigma(m+k+j,m+k+l)-(l-j).
\end{align*}
Along with \eqref{eq:2b-first-half}, this yields \eqref{eq:ternary-inequality-sigma-form}.
\end{proof}

The proof of Theorem \ref{thm:general-bound} uses the following lemma, whose easy proof is left as an exercise for the interested reader.

\begin{lemma} \label{lem:power-of-b}
For each $k\in\NN$, there exist integers $p\geq 0$ and $j\leq[(b+1)/2]$ such that $k\leq jb^p<2k$.
\end{lemma}

\begin{proof}[Proof of Theorem \ref{thm:general-bound}]
Fix $m$. We use induction on $k$. The statement is trivial when $k=0$, so let $1\leq l\leq m$ and assume \eqref{eq:general-bound} holds for all $k<l$, with strict inequality in case $b$ is odd. By Lemma \ref{lem:power-of-b}, there exist integers $p\geq 0$ and $j\leq[(b+1)/2]$ such that $l\leq jb^p<2l$. Thus for each $r\in\{m-l,\dots,m+l-jb^p-1\}$, we have $r<m$, $r+jb^p\in\{m,\dots,m+l-1\}$ and, by \eqref{eq:adding-r-powers}, $s_b(r+jb^p)\leq s_b(r)+j$. Hence,
\begin{align}
\Sigma(m-l+jb^p,m+l)-\Sigma(m-l,m+l-jb^p)&\leq j(2l-jb^p) \notag \\
&\leq \left[\frac{b+1}{2}\right](2l-jb^p).
\label{eq:direct-part}
\end{align}
And the induction hypothesis applied to $k=jb^p-l$ gives
\begin{equation}
\Sigma(m,m-l+jb^p)-\Sigma(m+l-jb^p,m)\leq \left[\frac{b+1}{2}\right](jb^p-l),
\label{eq:induction-part}
\end{equation}
since $jb^p-l<l$. Adding inequalities \eqref{eq:direct-part} and \eqref{eq:induction-part} yields 
\begin{equation}
\Sigma(m,m+l)-\Sigma(m-l,m)\leq \left[\frac{b+1}{2}\right]l,
\label{eq:final-result}
\end{equation}
so \eqref{eq:general-bound} holds also for $k=l$. Statement (ii) of the theorem follows immediately from \eqref{eq:general-bound} by rearranging terms and dividing by $2k$.

We next demonstrate strict inequality when $b$ is odd. Assume first that $l$ is of the form $l=jb^p$. (This includes the case $l=1$.) If $j<(b+1)/2$ we have strict inequality in \eqref{eq:direct-part}, so assume that $j=(b+1)/2$. But then $2l>b^{p+1}$, so we can replace $p$ with $p+1$ and $j$ with $1$ in the induction argument above, and once again obtain strict inequality in \eqref{eq:direct-part}, since $1<(b+1)/2$ for odd $b$.

When $l$ is not of the form $jb^p$, the induction hypothesis is used with $k=jb^p-l>0$, giving strict inequality in \eqref{eq:induction-part}. Thus, in both cases, we have strict inequality in \eqref{eq:final-result}.

Finally, we show that the inequality is sharp. For even $b$, take $m=k=b^n/2$ for any $n\in\NN$. It is easy to calculate inductively, using Lemma \ref{lem:base-multiply}, that
\begin{equation*}
S_b(b^n)-2S_b\left(\frac{b^n}{2}\right)=\frac{b^{n+1}}{4}=\left[\frac{b+1}{2}\right]\cdot\frac{b^n}{2},
\end{equation*} 
obtaining equality in \eqref{eq:general-bound}. (The base case $n=1$ is left as an exercise for the interested reader.) When $b$ is odd, the computation is more tedious. Here we take $m=k=k_n:=(b^n-1)/2$, and claim that
\begin{equation}
S_b(2k_n)-2S_b(k_n)=\frac{b+1}{2}k_n-\frac{(b-1)n}{2},
\label{eq:difference-formula}
\end{equation}
so that
\begin{equation*}
\frac{S_b(2k_n)-2S_b(k_n)}{k_n}\to \frac{b+1}{2}=\left[\frac{b+1}{2}\right], \qquad\mbox{as $n\to\infty$}.
\end{equation*}
To derive \eqref{eq:difference-formula} we start with the well-known observation that, for any $b$,
\begin{equation*}
S_b(b^n)=\frac{nb^n(b-1)}{2}, \qquad n\in\NN.
\end{equation*}
From this, we obtain
\begin{equation}
S_b(2k_n)=S_b(b^n)-s_b(b^n-1)=\frac{nb^n(b-1)}{2}-n(b-1).
\label{eq:twice-k}
\end{equation}
The computation of $S_b(k_n)$ may be done inductively, using the recursion $k_{n+1}=bk_n+(b-1)/2$. For $0\leq m<b$ and $k\in\NN$ we have
\begin{equation*}
S_b(bk+m)=S_b(bk)+\sum_{j=0}^{m-1}s_b(bk+j),
\end{equation*}
and since $s_b(bk+j)=s_b(k)+j$ for $0\leq j<b$, this leads via Lemma \ref{lem:base-multiply} to a recursion for $S_b(k_n)$, noting that $s_b(k_n)=n(b-1)/2$. One can then inductively verify the formula
\begin{equation*}
S_b(k_n)=\frac{b^n-1}{4}\left(n(b-1)-\frac{b+1}{2}\right).
\end{equation*}
This, together with \eqref{eq:twice-k}, leads after some more manipulations to \eqref{eq:difference-formula}.
\end{proof}

To prove Theorem \ref{thm:b-times-longer}, we will demonstrate a slightly stronger result. Define a partial order $\prec_b$ on $\NN$ by $n\prec_b m$ if and only if $\alpha_i(n;b)\leq \alpha_i(m;b)$ for every $i$.

\begin{theorem} \label{thm:b-ary-array}
Fix $b\geq 2$. For each $k\in\NN$, the numbers $0,1,\dots,bk-1$ can be arranged in a $b\times k$ matrix $A_k=[a_{i,j}]_{i=1,j=1}^{b,k}$ such that:
\begin{enumerate}[(i)]
\item $a_{1,j}=j-1$ for $j=1,\dots,k$;
\item $a_{1,j}\prec_b a_{i,j}$ for $i=1,\dots,b$ and $j=1,\dots,k$; and
\item $a_{i,j}-a_{1,j}$ is the sum of exactly $i-1$ powers of $b$; that is, $s_b(a_{i,j})=s_b(a_{1,j})+i-1$, for $i=1,\dots,b$ and $j=1,\dots,k$.
\end{enumerate}
\end{theorem}

An example of such an arrangement for $b=3$ and $k=5$ is
\begin{equation*}
A_5=\begin{bmatrix} 0 & 1 & 2 & 3 & 4\\9 & 10 & 11 & 6 & 5\\12 & 13 & 14 & 7 & 8 \end{bmatrix}.
\end{equation*}
Note that the arrangement is by no means unique: in the above example we could interchange $6$ and $12$, or $5$ and $7$, etc.

We prove Theorem \ref{thm:b-ary-array} by describing a simple algorithm for constructing the matrix $A_k$. This requires some terminology and a lemma. Fix $b\geq 2$. Suppose a finite set of pegs are placed in a finite rectangular array of holes. A hole has {\em position} $(i,j)$ if it is the $j$th hole (from the left) in the $i$th row (from the top). For $k\in\ZZ_+$, a $b^k$-{\em shift} is the move of a peg from any position $(i,j)$ with $j>b^k$ to the new position $(i+1,j-b^k)$. In other words, a $b^k$-shift moves a peg $b^k$ columns to the left and one row down. A {\em power shift} is any $b^k$-shift, where $k\in\ZZ_+$. A $b^k$-shift from $(i,j)$ to $(i+1,j-b^k)$ is {\em permissible} if position $(i+1,j-b^k)$ is not yet occupied and there is $l\in\NN$ such that $(l-1)b^{k+1}<j-b^k<j\leq lb^{k+1}$.

\begin{lemma} \label{lem:arrange-pegs}
For any $n\in\NN$, a single row of $n$ pegs can be rearranged by a finite sequence of permissible power shifts into a table of $b$ rows and $\lceil n/b\rceil$ columns so that each column except possibly the last contains $b$ pegs, and in the last column no peg is placed below an empty hole.
\end{lemma}

\begin{proof}
For $k\in\ZZ_+$, let a $k$-{\em tableau} be an arrangement of $b$ rows of pegs (possibly empty), aligned on the left and ordered by decreasing length, with the property that each row except perhaps one contains either zero or $b^k$ pegs. We claim that any $k$-tableau can be arranged by permissible power shifts into a table as described in the lemma. This is trivial for $k=0$, as a $0$-tableau already has the required form. Suppose the claim is true for some arbitrary $k\in\ZZ_+$, and let a $(k+1)$-tableau be given. Then some number $f\geq 0$ of rows (at the top of the table) contain $b^{k+1}$ pegs, row $f+1$ contains some number $m$ of pegs ($0\leq m<b^{k+1}$), and the remaining $b-f-1$ rows are empty. Note that in this tableau all $b^k$-shifts to empty holes are permissible.

Let $l\in\{1,\dots,b\}$ be such that $(l-1)b^k\leq m<lb^k$. After performing all permissible $b^k$-shifts, the tableau is transformed into a new table with:
\begin{itemize}\setlength{\itemsep}{0em}
\item $l-1$ rows of $(f+1)b^k$ pegs; followed by
\item one row of $fb^k+m-(l-1)b^k$ pegs; followed by
\item $b-l$ rows of $fb^k$ pegs.
\end{itemize}
In this new table, each row is at least $fb^k$ long, and columns $fb^k+1,\dots,(f+1)b^k$ form a $k$-tableau, which by the induction hypothesis can be rearranged as required. Together with the first $fb^k$ columns, this gives a rearrangement of the entire $(k+1)$-tableau as required.

The statement of the lemma now follows because a row of $n\geq 2$ pegs can be trivially turned into a $k$-tableau by adding $b-1$ empty rows, where $k$ is the integer such that $b^{k-1}<n\leq b^{k}$.
\end{proof}

\begin{proof}[Proof of Theorem \ref{thm:b-ary-array}]
We may apply Lemma \ref{lem:arrange-pegs} with $n=bk$ to see that the single row containing the numbers $0,1,\dots,bk-1$ in increasing order may be rearranged into a $b\times k$ matrix $[a_{i,j}]$ by permissible power shifts only. Clearly, the first row of this matrix contains the numbers $0,1,\dots,k-1$ in increasing order (since no numbers are ever moved into the first row by power shifts), so (i) is satisfied. We show (ii) by induction on $i$. Note that (ii) is trivial for $i=1$. Fix $j\in\{1,\dots,k\}$, and suppose $a_{1,j}\prec_b a_{i,j}$. The number $a_{i+1,j}$ was last moved from a position in row $i$ by shifting it some distance $b^r$ to the left. Since this was a permissible move, we have $j-1\prec_b j+b^r-1$, in other words, $a_{1,j}\prec_b a_{1,j+b^r}$. But $a_{i+1,j}$ had arrived at its position in row $i$ by a sequence of permissible moves, so by the induction hypothesis, $a_{1,j+b^r}\prec_b a_{i+1,j}$. Hence, $a_{1,j}\prec_b a_{i+1,j}$. This proves (ii). Property (iii) follows from (ii), as clearly $a_{i,j}-a_{1,j}$ is a sum of $i-1$ powers of $b$.
\end{proof}

\begin{proof}[Proof of Theorem \ref{thm:b-times-longer}]
Observe that, in terms of the notation $\Sigma_b(s,t)$, we are to prove that
\begin{equation}
\Sigma_b(n,n+bk)\leq b\Sigma_b(n,n+k)+\frac{b(b-1)}{2}k.
\label{eq:b-times-longer}
\end{equation}
Let $[a_{i,j}]_{i=1,j=1}^{b,k}$ be a matrix satisfying the conclusion of Theorem \ref{thm:b-ary-array}. Note that
\begin{equation*}
\Sigma_b(n,n+k)=\sum_{j=1}^k s_b(n+j-1)=\sum_{j=1}^k s_b(a_{i,j}+n),
\end{equation*}
and
\begin{equation*}
\Sigma_b(n,n+bk)=\sum_{j=1}^{bk}s_b(n+j-1)=\sum_{i=1}^b \sum_{j=1}^k s_b(a_{i,j}+n).
\end{equation*}
By property (iii) of Theorem \ref{thm:b-ary-array} and \eqref{eq:adding-r-powers}, $s_b(a_{i,j}+n)-s_b(a_{1,j}+n)\leq i-1$. Hence,
\begin{align*}
\Sigma_b(n,n+bk)&\leq \sum_{i=1}^b \sum_{j=1}^k\{s_b(a_{1,j}+n)+i-1\}\\
&=b\sum_{j=1}^k s_b(a_{1,j}+n)+k\sum_{i=1}^b (i-1)\\
&=b\Sigma_b(n,n+k)+\frac{b(b-1)}{2}k,
\end{align*}
completing the proof.
\end{proof}

Note that property (ii) of Theorem \ref{thm:b-ary-array} was not needed in the last proof. However, dropping the requirement (ii) from Theorem \ref{thm:b-ary-array} does not appear to lead to a simpler proof, whereas including it adds to the independent interest of that theorem.

\section{Application to approximate convexity} \label{sec:application}

Delange \cite{Delange} introduced the functions
\begin{equation*}
h_b(x)=\sum_{n=0}^\infty b^{-n}g_b(b^n x),
\end{equation*}
where for each $b\geq 2$, $g_b$ is the 1-periodic continuous function defined on $[0,1)$ by
\begin{equation*}
g_b(x)=\int_0^x\left(\frac{b-1}{2}-[bt]\right)\,dt.
\end{equation*}
For the case $b=2$, we have $g_2(x)=(1/2)\dist(x,\ZZ)$, where $\dist(x,\ZZ)$ denotes the distance from $x$ to the nearest integer, and hence $h_2$ is one-half times the Takagi function \cite{Takagi}. The relationship between the Takagi function and the binary digital sum $S_2$ was first established by Trollope \cite{Trollope}. Delange \cite{Delange} generalized this relationship by showing that, for each $n\in\NN$,
\begin{equation}
S_b(n)=\frac{b-1}{2}n\log_b n+nF(\log_b n),
\label{eq:Delange}
\end{equation}
where
\begin{equation*}
F(x)=\frac{b-1}{2}(1-\{x\})-b^{1-\{x\}}h_b(b^{\{x\}-1}),
\end{equation*}
in which $\{x\}:=x-[x]$ denotes the fractional part of $x$. (The function $h$ in Delange's paper is actually $-h_b$; the reason for the present representation is that $h_b$ is actually nonnegative, as is easily verified.) In addition to establishing \eqref{eq:Delange}, Delange \cite{Delange} proves that $h_b$ is nowhere differentiable for each $b\geq 2$.

A different sequence of functions was recently introduced by Lev \cite{Lev}. 
For $b\in\NN$, let $\phi_b(x)=\min\{\dist(x,\ZZ),1/b\}$, and define the function
\begin{equation*}
\omega_b(x):=\sum_{n=0}^\infty b^{-n}\phi_b(b^n x).
\end{equation*}
Lev demonstrates a direct connection between $\omega_b(x)$ and the edge-isoperimetric problem for Cayley graphs of homocyclic groups of exponent $b$. Comparison with Delange's functions shows that $\omega_2=2h_2$, and $\omega_3=h_3$. After that, the two sequences go their separate ways: For $b\geq 4$, there is no 
direct relationship between $h_b$ and $\omega_b$, although $\omega_4=(1/2)\omega_2=h_2$.

For the Takagi function $\omega_2$, Boros \cite{Boros} proved the inequality
\begin{equation}
\omega_2\left(\frac{x+y}{2}\right)\leq \frac{\omega_2(x)+\omega_2(y)}{2}+\frac{|y-x|}{2},
\label{eq:Boros-Pales}
\end{equation}
which had been conjectured by H\'azy and P\'ales \cite{Hazy-Pales}. We will show here that all of Delange's functions satisfy an inequality similar to \eqref{eq:Boros-Pales}.

\begin{theorem} \label{thm:Delange-approx-convex}
Let $b\geq 2$. For all real $x$ and $y$ with $x<y$, we have
\begin{equation}
h_b\left(\frac{x+y}{2}\right)\leq \frac{h_b(x)+h_b(y)}{2}+\frac14\left[\frac{b+1}{2}\right](y-x).
\label{eq:Delange-approx-convex}
\end{equation}
\end{theorem}

For $\omega_3=h_3$, Lev \cite[Theorem 3]{Lev} proves the following interesting inequality, which develops the Boros-Pales inequality in a different but equally natural direction.

\begin{theorem}[Lev, 2012] \label{thm:lev}
For all real $x$, $y$ and $z$ with $x\leq y\leq z$, we have
\begin{equation}
h_3\left(\frac{x+y+z}{3}\right)\leq \frac{h_3(x)+h_3(y)+h_3(z)}{3}+\frac13(z-x).
\label{eq:ternary-approximate-convexity}
\end{equation}
\end{theorem}

It is straightforward to deduce Theorems \ref{thm:Delange-approx-convex} and \ref{thm:lev} from Theorems \ref{thm:general-bound} and \ref{thm:ternary-inequality}, respectively. The key is to derive an expression for $h_b$ at points of the form $x=k/b^n$ in terms of $S_b$, and to use the continuity of $h_b$.

We first define the partial sums
\begin{equation*}
h_b^{(n)}(x)=\sum_{k=0}^{n-1} b^{-k}g_b(b^k x),
\end{equation*}
and note that for $k\in\ZZ$, $h_b(k/b^n)=h_b^{(n)}(k/b^n)$.
For $x\in[0,1)$, let $x=\sum_{i=1}^\infty \eps_i(x)b^{-i}$ denote the $b$-ary expansion of $x$, where $\eps_i(x)\in\{0,1,\dots,b-1\}$. If $x$ is of the form $x=k/b^n$, we take the expansion ending in all zeros. Observe that for each $x\in(0,1)$, the right-hand derivative of $g_b$ at $x$ is $(b-1)/2-\eps_1(x)$.
Hence, by the periodicity of $g_b$, the slope of $h_b^{(n)}$ at any point $x$ not of the form $k/b^n$ is
\begin{equation*}
\sum_{k=1}^n\left(\frac{b-1}{2}-\eps_k(x)\right)=\frac{b-1}{2}n-\sum_{k=1}^n \eps_k(x). 
\end{equation*}
This simple observation yields the formula
\begin{equation*}
h_b\left(\frac{k}{b^n}\right)-h_b\left(\frac{k-1}{b^n}\right) =b^{-n}\left(\frac{b-1}{2}n-s_b(k-1)\right),
\end{equation*}
and hence,
\begin{equation}
h_b\left(\frac{k}{b^n}\right)= b^{-n}\left(\frac{b-1}{2}kn -\sum_{i=0}^{k-1}s_b(i)\right) =b^{-n}\left(\frac{b-1}{2}kn-S_b(k)\right).
\label{eq:b-ary-expression}
\end{equation}

\begin{proof}[Proof of Theorem \ref{thm:Delange-approx-convex}]
Assume first that there exist nonnegative integers $n$, $m$ and $k$ such that
\begin{equation}
x=\frac{m-k}{b^n}, \qquad y=\frac{m+k}{b^n},
\label{eq:b-adic-rational}
\end{equation}
so that $(x+y)/2=m/b^n$.
One verifies easily using \eqref{eq:b-ary-expression} that
\begin{equation*}
2h_b\left(\frac{m}{b^n}\right)-h_b\left(\frac{m-k}{b^n}\right)-h_b\left(\frac{m+k}{b^n}\right)=\frac{S_b(m-k)+S_b(m+k)-2S_b(m)}{b^n},
\end{equation*}
since the terms involving $(b-1)/2$ cancel. Thus, Theorem \ref{thm:general-bound} gives \eqref{eq:Delange-approx-convex} for $x$ and $y$ of the form \eqref{eq:b-adic-rational}, as $k/b^n=(y-x)/2$. But any two real points $x$ and $y$ with $x<y$ can be approximated arbitrarily closely by points $x'$ and $y'$ of the form \eqref{eq:b-adic-rational}. Thus, the proof is completed by using the continuity of $h_b$.
\end{proof}

\begin{proof}[Proof of Theorem \ref{thm:lev}]
Let $0\leq x\leq y\leq z$, and put $a=(x+y+z)/3$. By symmetry of $h_3$, we may assume without loss of generality that $y\leq (x+z)/2$, so that $x\leq y\leq a\leq z$. Since $h_3$ is continuous, we may assume further that $x,y,z$ and $a$ are all triadic rational; that is, there exist nonnegative integers $n, m, k$ and $l$ with $m\geq k\geq l$ such that
\begin{equation*}
a=\frac{m}{3^n}, \qquad x=\frac{m-k}{3^n}, \qquad y=\frac{m-l}{3^n}, \qquad z=\frac{m+k+l}{3^n}.
\end{equation*}
Upon multiplying both sides by $3$, we can write \eqref{eq:ternary-approximate-convexity} for this case as
\begin{equation*}
3h_3\left(\frac{m}{3^n}\right)\leq h_3\left(\frac{m-k}{3^n}\right) +h_3\left(\frac{m-l}{3^n}\right) +h_3\left(\frac{m+k+l}{3^n}\right) +\frac{2k+l}{3^n}.
\end{equation*}
By \eqref{eq:b-ary-expression}, this is equivalent to
\begin{multline*}
3(mn-S_3(m))\leq [(m-k)n-S_3(m-k)]+[(m-l)n-S_3(m-l)]\\
+[(m+k+l)n-S_3(m+k+l)]+(2k+l),
\end{multline*}
and this simplifies to \eqref{eq:ternary-inequality}.
\end{proof}

Note that the number $n$ disappears from the inequality in the end. This suggests that Lev's approach of induction on $n$ is perhaps not the most natural. While the above proof uses Theorem \ref{thm:ternary-inequality}, whose proof is quite long, Lev's original proof is rather lengthy as well, and the present proof seems to be conceptually more pleasing.

\end{document}